\documentclass{commat}

\theoremstyle{definition}

\theoremstyle{remark}

\numberwithin{equation}{section}

\title[Solvable Leibniz superalgebras with filiform nilradical]{%
    Solvable Leibniz superalgebras whose nilradical has the characteristic sequence $(n-1,1  \ | \ m)$ and nilindex $n+m$
    }

\author{%
    Abror Khudoyberdiyev, Khosiyat Muratova.
    }

\authorinfo[%
    A.Kh.Khudoyberdiyev]{
    Institute of Mathematics Uzbekistan Academy of Science. National University of Uzbekistan}{%
    khabror@mail.ru
    }

\authorinfo[%
    Kh.A. Muratova]{
    Kimyo International University in Tashkent, Uzbekistan. Institute of Mathematics Uzbekistan Academy of Science}{%
    xalkulova@gmail.com
    }

\abstract{%
    Leibniz superalgebras with nilindex $n + m$ and characteristic sequence $(n-1, 1 \ | \ m)$  divided into four parametric classes that contain a set of non-isomorphic superalgebras. In this paper, we give a complete classification of solvable Leibniz superalgebras whose nilradical is a nilpotent Leibniz superalgebra with nilindex $n + m$ and characteristic sequence $(n-1, 1 \ | \ m)$.
We obtain a condition for the value of parameters of the classes of such nilpotent superalgebras for which they have a solvable extension. Moreover, the classification of solvable Leibniz superalgebras whose nilradical is a Lie superalgebra with the maximal nilindex is given.
    }



\VOLUME{32}
\YEAR{2024}
\NUMBER{2}
\firstpage{27}
\DOI{https://doi.org/10.46298/cm.11369}

\begin{document}

\section{Introduction}
The theory of supervariety and superalgebras is one of the important direction of modern mathematics which generalizes many objects from differential and algebraic geometry.
The interest in superalgebras is explained by their ability to unify bosons and fermions in physics,
to integrate the groups of internal and dynamic symmetries into one complex, and to transfer all fundamental powers into a unified field.
Lie supergroups and superalgebras are the most widely used supermanifolds in theory of supersymmetry,
and the structure theory and classification problems of Lie superalgebras are important problems in non-associative algebras.
Lie superalgebras are generalizations of most important object of Lie algebras and for many years they attract
the attention of both the mathematicians and physicists \cite{Berezin, Dusan, Gilg, Kac, Lietes}.

In Lie theory, there are many works devoted to the study of finite-dimensional solvable and nilpotent Lie algebras.
First, we recall that in 1945 it was proved by A.I. Malcev that a solvable Lie algebra is fully determined by its nilradical
\cite{Malcev}. Further, in 1963, Mubarakzjanov developed a method for constructing solvable algebras using the nilradical and its nil-independent derivations \cite{Muba}.
Using this method, a number of solvable Lie algebras with given nilradicals were constructed, such as: abelian, Heisenberg, filiform, quasi-filiform algebras and others \cite{AnCaGa3, An2, Ngomo, RubWint, Snobl, Wang}.

The notion of Leibniz algebras was introduced in \cite{Loday} as
a non-antisymmetric generalization of Lie algebras.
In recent years it has been a common theme to extend various results from Lie algebras to Leibniz algebras \cite{AyupovBook}.
Specifically, variations of Engel's theorem for Leibniz algebras have been proved by different authors and D. Barnes proved Levi's theorem for Leibniz algebras \cite{Bar}.
The analog of Mubarakzjanov's result has been applied to Leibniz
algebras case in \cite{Casas}, which shows the importance of the
consideration of their nilradicals in the Leibniz algebra case as
well. The papers \cite{AlvKay, Kay1, Kay2, Kay3, LinBos, Gay, KhLO, KhRSH} are also devoted
to the algebraic and geometric classification of some important classes of finite-dimensional Leibniz algebras.

Leibniz superalgebras are generalizations of Leibniz algebras, and on the other hand, they
naturally generalize Lie superalgebras.
Leibniz superalgebras first were considered in \cite{Livernet} under the name of graded Leibniz algebras.
The concept of Leibniz superalgebra and its
cohomology was first introduced by Dzhumadil'daev in
\cite{dzhumadil}. The term Leibniz superalgebras was first used in \cite{Alb}, whose nilpotent Leibniz
superalgebras with the maximal index of nilpotency were classified. It should be noted that Lie superalgebras with
maximal nilindex were classified in \cite{GomezKhak}.
The distinctive property of  Leibniz superalgebra is that the maximal nilindex of $(n+m)$-dimensional Leibniz superalgebra is equal to $n+m+1$.
The description of Leibniz superalgebras with dimensions of even and odd parts equal to $n$ and $m$, respectively, and with nilindex $n + m$ were classified
by applying restrictions the invariant called characteristic sequences in \cite{FilSup, FS, C-G-O-Kh, G-O-Kh}.

The next natural step in the theory of finite-dimensional Leibniz (Lie) superalgebras is to extend the method of classification of solvable superalgebras with their nilradical.
It should be noted that the structures of solvable Lie and Leibniz superalgebras are more complex than the structures of solvable Lie and Leibniz algebras. In particular, an analog of Lie's theorem is not yet known even for Lie superalgebras. However, the corollary to the Lie theorem that the square of a solvable algebra is nilpotent is not true for Lie superalgebras, that is, in \cite{Dusan} an example was constructed of a solvable Lie superalgebra whose square is not nilpotent.
Despite all the difficulties, in \cite{Nav, CNO2} the solvable extension method for Leibniz superalgebras was established.
The papers \cite{Bull, salg} are also devoted to the description of solvable Lie and Leibniz superalgebras.
In particular, in the papers \cite{Nav, Bull}, solvable Lie and Leibniz superalgebras are obtained for whose nilradical has the maximal index of nilpotency.
Some solvable Leibniz superalgebras for whose nilradical has the nilindex $n+m$ were classified in \cite{UzMat, MIB}.
In this paper, we give a classification of solvable Leibniz superalgebras whose nilradical has the index of nilpotency $n + m$ and characteristic sequence $(n-1, 1 | m)$.
Specific values are obtained for the parameters of the classes of such nilpotent superalgebras for which they have a solvable extension.

\section{Preliminaries}

In this section, we give definitions and facts concerning Lie and Leibniz's superalgebras.
\begin{definition} A $\mathbb{Z}_2$-graded vector space $G=G_0\oplus G_1$  is called a Lie superalgebra if it is equipped with a product $[-,-]$ which satisfies the following conditions:

1. $[x,y]=-(-1)^{\alpha\beta}[y,x],$ for any $x\in G_\alpha,$
$y\in G_\beta,$

2. $(-1)^{\alpha\gamma}
[x,[y,z]]+(-1)^{\alpha\beta}[y,[z,x]]+(-1)^{\beta\gamma}[z,[x,y]]=0$

for any $x\in G_\alpha,$ $y\in G_\beta,$ $z\in G_\gamma$(Jacobi superidentity).
\end{definition}

\begin{definition} A $\mathbb{Z}_2$-graded vector space $L=L_0\oplus L_1$ is called a Leibniz superalgebra if it is equipped with a product $[-,-]$ which satisfies the following condition:
\[ \big[x,[y,z]\big]=\big[[x,y],z\big] - (-1)^{\alpha \beta}\big[[x,z],y\big]- \mbox{Leibniz superidentity} \] for all $x\in L, y\in L_{\alpha}, z\in L_{\beta}.$
\end{definition}

The vector spaces $L_0$ and $L_1$ are said to
be the even and odd parts of the superalgebra $L$,
respectively. It is obvious that $L_0$ is a Leibniz algebra and $L_1$ is a representation of $L_0.$
Note that if in Leibniz superalgebra $L$ the identity
$$[x,y]=-(-1)^{\alpha\beta} [y,x]$$ holds for any $x \in
L_{\alpha}$ and $y \in L_{\beta},$ then the Leibniz superidentity
can be transformed into the Jacobi superidentity. Thus, Leibniz
superalgebras are generalization of Lie superalgebras.

The notions of nilpotency and solvability of Leibniz superalgebras are defined in the same way as for Leibniz algebras.
For solvable Leibniz superalgebras we have that a Leibniz superalgebra $L$ is solvable if and only if its Leibniz algebra $L_0$ is solvable.
The concept of derivations of superalgebras differs from the notion of derivations of algebras, and as in a $\mathbb{Z}_2$-graded algebra, the space of derivations consists of even and odd subspaces.
Recall, now the definition of superderivations of Leibniz superalgebras \cite{Kac, KhudOmirov}.

\begin{definition}\label{difdef} A superderivation (or derivation) of a superalgebra $L$ of degree $s$ is a linear map $D: L \rightarrow L$ satisfying the following condition:
\[D([x,y])=[D(x), y] + (-1)^{s\cdot \alpha} [x, D(y)],\]
where $x \in L_{\alpha}, y \in L$ and $s, \alpha \in \mathbb{Z}_2$
\end{definition}

For convenience, let us shorten "derivation of even degree" to just even derivation.
Linear operator $R_x : L\rightarrow L, \ x \in L$ such that $R_x(y)=(-1)^{\alpha \beta}[y,x],$ $x \in L_{\alpha}, y\in L_{\beta}$ is called a right multiplication operator.
It is known that such an operator is a derivation of the Leibniz superalgebra $L$ of degree $s$ for $x \in L_s.$

Engel's theorem and its direct consequences remain valid for Leibniz superalgebras. In particular, a Leibniz superalgebra $L$ is nilpotent if and only if $R_x$ is nilpotent for every homogeneous element $x$ of $L.$ Here is the definition of nil-independency of the superderivation of degree $s$ imitated from Lie case (see \cite{Muba}).

\begin{definition}
Let $d_1, d_2, \dots, d_n$ be derivations of a Leibniz superalgebra $L$ of degree $s.$ The derivations $d_1, d_2, \dots, d_n$ are said to be a linearly
nil-independent if for $\alpha_1, \alpha_2, \dots ,\alpha_n \in \mathbb{C}$
and a natural number $k$
$$(\alpha_1d_1 + \alpha_2d_2+ \dots +\alpha_nd_n)^k=0\ \hbox{implies }\ \alpha_1 = \alpha_2= \dots =\alpha_n=0.$$
\end{definition}

Note that the maximal nilpotent ideal $N$ of the Leibniz superalgebra $L$ such that $[L, L] \subset N$ is called a nilradical.
In \cite{CNO2} important results regarding the solvable extension method for the finite-dimensional case are given and
it is shown that solvable Lie and Leibniz superalgebras can be described using nil-independent even derivations of the nilradical.
Additionally, it is proved that the dimension of a solvable Leibniz superalgebra with a given nilradical is bounded by the maximal number of nil-independent even derivations of the nilradical.

\begin{definition} The set $$\operatorname{Ann}_r(L)=\left\{ z\in L\ |\ [L,
z]=0\right\}$$ is called the right annihilator  of the superalgebra
$L.$
\end{definition}

Note that, elements of the form $[a,b]+(-1)^{\alpha \beta}[b,a],$ ($a \in
L_{\alpha}, \ b \in L_{\beta}$) are contained in $\operatorname{Ann}_r(L)$.

Let $L=L_0\oplus L_1$ be a nilpotent Leibniz superalgebra.
Operator $R_x, x \in L_0$ is a nilpotent endomorphism of the space $L_i$, where $i\in \{0,1\}$.
Taking into account the property of complex field we can consider the Jordan form of $R_x$. Denote by $C_i(x)(i\in \{0,1\})$ the descending sequence of the Jordan
blocks with dimensions of $R_x$. Consider the lexicographical order on the set $C_i(L_0)$.

\begin{definition}  A sequence
$$C(L)=\big(\mathop {\max
}\limits_{x \in L_0\setminus L_0^2}C_0(x)| \mathop {\max
}\limits_{\widetilde{x} \in L_0\setminus L_0^2}C_1(\widetilde{x}) \big)$$
is said to be the characteristic sequence of the Leibniz superalgebra $L$.
\end{definition}


%
%

\section{Solvable Leibniz superalgebras whose nilradical is a Lie superalgebra with maximal nilindex}

In this section, we give the description of solvable Leibniz superalgebras whose nilradical is a Lie superalgebra with the maximal index of nilpotency.
 Note that $(n+m)$-dimensional Lie superalgebra with nilindex $n+m$ exists only for $n=2$, $m$ is odd and the multiplication table of a superalgebra is as follows:
\[N_{2,m}: \left\{ \begin{array}{ll} [y_i,e_1]=y_{i+1}, & 1\leq i\leq m-1, \\[1mm] [y_{m+1-i},y_i]=(-1)^{i+1}e_2, & 1\leq i\leq \frac{m+1}{2}.\end{array}\right.\]

In the next theorems, we present the classification of solvable Leibniz (Lie) superalgebras with the nilradical $N_{2,m},$ which implies using the results of
the works \cite{Nav} and \cite{Bull}.
It should be noted that here we give list of solvable Lie superalgebras after some minor corrections and add the list of solvable non-Lie Leibniz superalgebras.

\begin{theorem}
 Let $L=L_0\oplus L_1$ be a $(m+3)$-dimensional solvable Leibniz superalgebra whose nilradical is isomorphic to $N_{2,m}$.
Then $L$ is isomorphic to one of the following pairwise non-isomorphic superalgebras:
$$M_1:\begin{cases}
[e_1,x]=-[x,e_1]=e_1,\\[1mm]
[x,x]=e_2,\\[1mm]
[y_i,e_1]=-[e_1,y_i]=y_{i+1},& 1\leq i\leq m-1,\\[1mm]
[y_{m+1-i},y_i]=-[y_i,y_{m+1-i}]=(-1)^{i+1}e_2,& 1\leq i\leq \frac{m+1}{2},\\[1mm]
[y_i,x]=-[x,y_i]=(i-\frac{m+1}{2})y_i,& 1\leq i\leq m.
\end{cases}$$
\[M_2(\alpha):\left\{\begin{array}{lll}
[e_1,x]=-[x,e_1]=e_1,\\[1mm]
[e_2,x]= -[x,e_2]=\alpha e_2,& \\[1mm]
[y_i,e_1]=-[e_1,y_i]=y_{i+1},& 1\leq i\leq m-1,\\[1mm]
[y_{m+1-i},y_i]=-[y_i,y_{m+1-i}]=(-1)^{i+1}e_2,& 1\leq i\leq \frac{m+1}{2},\\[1mm]
[y_i,x]=-[x,y_i]=(i+\frac{\alpha-m-1}{2}) y_i,& 1\leq i\leq m,
\end{array}\right.\]
\[M_3:\left\{\begin{array}{lll}
[e_1,x]=-[x,e_1]=e_1+e_2,\\[1mm]
 [e_2,x]= -[x,e_2]=e_2,\\[1mm]
[y_i,e_1]=-[e_1,y_i]=y_{i+1},& 1\leq i\leq m-1,\\[1mm]
[y_{m+1-i},y_i]=-[y_i,y_{m+1-i}]=(-1)^{i+1}e_2,& 1\leq i\leq \frac{m+1}{2},\\[1mm]
[y_i,x]=-[x,y_i]=(i-\frac{m}{2})y_i,& 1\leq i\leq m,\\[1mm]
\end{array}\right.\]
\[ M_4(b_2,b_4,\dots,b_{m-1}): \left\{\begin{array}{lll}
[e_2,x]=-[x,e_2]=2e_2,\\[1mm]
[y_i,e_1]=-[e_1,y_i]=y_{i+1},& 1\leq i\leq m-1,\\[1mm]
[y_{m+1-i},y_i]=-[y_i,y_{m+1-i}]=(-1)^{i+1}e_2,&1\leq i \leq \frac{m+1}{2},\\[1mm]
[y_i,x]= -[x,y_i]=y_i+\sum\limits_{k=1}^{[\frac{m-i+1}{2}]}b_{2k}y_{i+2k-1}, &1\leq i\leq m,
\end{array}\right.\]
Note that the first nonzero parameter of the algebra $M_4(b_2,b_4,\dots,b_{m-1})$ can be reduced to 1.
\end{theorem}

\begin{theorem}
 Let $L=L_0\oplus L_1$ be a $(m+4)$-dimensional solvable Leibniz superalgebra whose nilradical is isomorphic to $N_{2,m}$.
Then $L$ is isomorphic to the following Lie superalgebra:
\[M_5: \left\{\begin{array}{lll}
[e_1,x]=-[x,e_1]=e_1,\\[1mm]
[e_2,x]=-[x,e_2]=(m-1)e_2,\\[1mm]
[e_2,z]=-[z,e_2]=2e_2,\\[1mm]
[y_i,e_1]=-[e_1,y_i]=y_{i+1},& 1\leq i\leq m-1,\\[1mm]
[y_{m+1-i},y_i]=-[y_i,y_{m+1-i}]=(-1)^{i+1}e_2,&1\leq i \leq \frac{m+1}{2},\\[1mm]
[y_i, x]=-[x,y_i]=(1-i)y_i,& 1\leq i\leq m,\\[1mm]
[y_i, z]=-[z,y_i]=y_i,& 1\leq i\leq m.
\end{array}\right.\]
\end{theorem}

\section{Main result}

In this section, we give the classification of solvable Leibniz superalgebras whose nilradical has nilindex $n+m$ and characteristic sequence $(n-1,1 | m).$
Note that, in the case of $n=m=2$ there are two four-dimensional Leibniz superalgebras of nilindex four \cite{FilSup}. Solvable Leibniz superalgebras with these four-dimensional nilradicals are classified in \cite{MIB}. Thus, we consider the case $n\geq 3$ and in the following theorem we give the list of nilpotent Leibniz superalgebras with nilindex $n+m$ and the characteristic sequence $(n-1,1|m)$.

\begin{theorem}[\cite{FilSup}] \label{L}  Let $L$ be a Leibniz superalgebra of nilindex $n+m$ with characteristic sequence $(n-1,1|m)$, then $m=n-1$ or $m=n$ and there exists such a basis $\{e_1,e_2,\dots, e_n, y_1,y_2,\dots,y_{m}\}$
in superalgebra $L$ whose multiplication in this basis is as follows:

if $m=n-1$,
 $L(\alpha_4, \alpha_5, \ldots, \alpha_n, \theta):$
\begin{equation}\label{eqL}
\left\{\begin{array}{ll}
[e_1,e_1]=e_3, \quad  [e_i,e_1]=e_{i+1},&    2 \le i \le n-1,
\\[1mm]
[y_j,e_1]=y_{j+1},    & 1 \le j \le n-2,
\\[1mm]
[e_1,y_1]= \frac12 y_2, \quad [e_i,y_1]= \frac12 y_i,  &    2 \le i \le n-1,
\\[1mm]
[y_1,y_1]=e_1, \quad 
[y_j,y_1]=e_{j+1},& 2 \le j \le n-1,
\\[1mm]
[e_1,e_2]=\alpha_4e_4+ \alpha_5e_5+ \ldots +
\alpha_{n-1}e_{n-1}+ \theta e_n,& 
\\[1mm]
[e_j,e_2]= \alpha_4e_{j+2}+ \alpha_5e_{j+3}+ \ldots +
\alpha_{n+2-j}e_n,&  2 \le j \le n-2,
\\[1mm]
[y_1,e_2]= \alpha_4y_3+ \alpha_5y_4+ \ldots +
\alpha_{n-1}y_{n-2}+\theta y_{n-1},&
\\[1mm]
[y_j,e_2]= \alpha_4y_{j+2}+ \alpha_5y_{j+3}+ \ldots +
\alpha_{n+1-j}y_{n-1},&  2 \le j \le n-3, \end{array} \right.\end{equation}

$G(\beta_4, \beta_5, \ldots, \beta_n, \gamma):$
\begin{equation}\label{eqG}
\left\{\begin{array}{lll}
[e_1,e_1]=e_3,& [e_i,e_1]=e_{i+1},&    3 \leq i \leq n-1,
\\[1mm]
[y_j,e_1]=y_{j+1},& &   1 \leq j \leq n-2,
\\[1mm]
[e_1,y_1]= \frac12 y_2,&
[e_i,y_1]= \frac12 y_i,  &   3 \leq i \leq n-1,
\\[1mm]
[y_1,y_1]=e_1,&
[y_j,y_1]=e_{j+1},& 2 \leq j \leq n-1,
\\[1mm]
[e_1,e_2]=\beta_4e_4+ \beta_5e_5+ \ldots +
\beta_{n}e_n,& &
\\[1mm]
[e_j,e_2]= \beta_4e_{j+2}+ \beta_5e_{j+3}+ \ldots +
\beta_{n+2-j}e_n,& [e_2,e_2]=\gamma e_n, & 3 \le j \le n-2,
\\[1mm]
[y_j,e_2]= \beta_4y_{j+2}+ \beta_5y_{j+3}+ \ldots +
\beta_{n+1-j}y_{n-1},& & 1 \le j \le n-3, \end{array} \right.\end{equation}

if $m=n$, then $M(\alpha_4, \alpha_5, \ldots, \alpha_n, \theta, \tau):$
\begin{equation}\label{eqM} \left\{
\begin{array}{lll}
[e_1,e_1]=e_3, \quad [e_i,e_1]=e_{i+1},&     2 \le i \le n-1,
\\[1mm]
[y_j,e_1]=y_{j+1},   &  1 \le j \le n-1,
\\[1mm]
[e_1,y_1]=  \frac12 y_2, \quad  [e_i,y_1]=  \frac12 y_i, &    2 \le i \le n,
\\[1mm]
[y_1,y_1]=e_1, \quad [y_j,y_1]=e_{j+1},& 2 \le j \le n-1,
\\[1mm]
[e_1,e_2]=\alpha_4e_4+ \alpha_5e_5+ \ldots +
\alpha_{n-1}e_{n-1}+ \theta e_n,&
\\[1mm]
 [e_j,e_2]= \alpha_4e_{j+2}+
\alpha_5e_{j+3}+ \ldots + \alpha_{n+2-j}e_n,& 2 \le j \le n-2,
\\[1mm]
[y_1,e_2]= \alpha_4y_3+  \ldots +
\alpha_{n-1}y_{n-2}+\theta y_{n-1}+\tau y_n,&
\\[1mm]
[y_2,e_2]= \alpha_4y_4+ \alpha_5y_4+ \ldots +
\alpha_{n-1}y_{n-1}+\theta y_n,&
\\[1mm]
[y_j,e_2]= \alpha_4y_{j+2}+ \alpha_5y_{j+3}+ \ldots +
\alpha_{n+2-j}y_{n},&  3 \le j \le n-2,\end{array} \right.\end{equation}

$H(\beta_4, \beta_5, \ldots, \beta_n, \delta, \gamma):$
\begin{equation}\label{eqH}
\left\{\begin{array}{lll}
[e_1,e_1]=e_3,& [e_i,e_1]=e_{i+1},&    3 \leq i \leq n-1,
\\[1mm]
[y_j,e_1]=y_{j+1},&    &1 \leq j \leq n-1,
\\[1mm]
[e_1,y_1]= \frac12 y_2,&
[e_i,y_1]= \frac12 y_i,  &   3 \leq i \leq n,
\\[1mm]
[y_1,y_1]=e_1,& [y_j,y_1]=e_{j+1},& 2 \leq j \leq n-1,
\\[1mm]
[e_1,e_2]=\beta_4e_4+ \beta_5e_5+ \ldots +
\beta_{n}e_n,&&
\\[1mm]
[e_j,e_2]= \beta_4e_{j+2}+ \beta_5e_{j+3}+ \ldots +
\beta_{n+2-j}e_n,& [e_2,e_2]=\gamma e_n,& 3 \le j \le n-2,
\\[1mm]
[y_1,e_2]= \beta_4y_{3}+ \beta_5y_{4}+ \ldots +
\beta_{n}y_{n-1}+\delta y_n,\\[1mm]
[y_j,e_2]= \beta_4y_{j+2}+ \beta_5y_{j+3}+ \ldots +
\beta_{n+2-j}y_{n},& & 2 \le j \le n-2. \end{array} \right.\end{equation}

\end{theorem}

First, we describe the even derivations of these nilpotent Leibniz superalgebras.
\begin{proposition}\label{difL}
An even derivation of $L(\alpha_4, \alpha_5,\dots, \alpha_n, \theta)$ has the following form:
$$\left\{\begin{array}{lll}
d(e_1)=2a_1e_1+a_2e_3+a_3e_4+\dots+a_{n-1}e_n,\\[1mm]
d(e_2)=2a_1e_2+a_2e_3+a_3e_4+\dots+a_{n-2}e_{n-1}+b_ne_n,\\[1mm]
d(e_i)=2(i-1)a_1e_i+a_2e_{i+1}+a_{3}e_{i+2}+\dots+a_{n-i+1}e_n, & 3\leq i\leq n,\\[1mm]
d(y_i)=(2i-1)a_1y_i+a_2y_{i+1}+a_{3}y_{i+2}+\dots+a_{n-i}y_{n-1}, & 1\leq i\leq n-1,\\[1mm]
\end{array}\right.$$
where $(n-3)\theta a_1=0, \quad \alpha_i a_1=0, \ 4\leq i\leq n.$
\end{proposition}

\begin{proof} Let $d$ be an even derivation of a Leibniz superalgebra which belongs to the class
$L(\alpha_4,\alpha_5, \dots, \alpha_n, \theta).$ Put $$d(y_1)=a_1y_1+a_2y_2+\dots+a_{n-1}y_{y-1}, \quad d(e_2)=b_1e_1+b_2e_2+\dots+b_ne_n.$$

Using the multiplications of the superalgebra and Definition \ref{difdef},  we find the following:
$$\begin{array}{lll}
d(e_1)=d([y_1,y_1])=[d(y_1),y_1]+[y_1,d(y_1)]=2a_1e_1+a_2e_3+a_3e_4+\dots+a_{n-1}e_n,\\[1mm]
d(e_3)=d([e_2,e_1])=[d(e_2),e_1]+[e_2,d(e_1)]=(b_1+b_2+2a_1)e_3+b_3e_4+\dots+b_{n-1}e_n,\\[1mm]
\end{array}$$

On the other hand,
$$d(e_3)=d([e_1,e_1])=[d(e_1),e_1]+[e_1,d(e_1)]=4a_1e_3+a_2e_4+a_3e_4+\dots+a_{n-2}e_n.$$

Comparing the coefficients at the basic elements, we get that
$$b_1+b_2=2a_1, \ b_i=a_{i-1}, \ 3\leq i\leq n-1.$$

Since $[e_k,e_1]=e_{k+1}$ for $3\leq k\leq n,$ then from $$d(e_{k+1})=d([e_k,e_1])=[d(e_k),e_1]+[e_k,d(e_1)],$$ we have
$$d(e_i)=2(i-1)a_1e_i+a_2e_{i+1}+a_3e_{i+2}+\dots+a_{n-i+1}e_n,\quad 3\leq i\leq n.$$

Now we consider
$$d(y_2)=d([y_1,e_1])=[d(y_1),e_1]+[y_1,d(e_1)]=3a_1y_2+a_2y_3+\dots+a_{n-2}y_{n-1}.$$

From $d(y_i)=d([y_{i-1},e_1])=[d(y_{i-1}), e_1]+[y_{i-1},d(e_1)],$ inductively we get
$$d(y_i)=(2i-1)a_1y_i+a_2y_{i+1}+\dots+a_{n-i}y_{n-1}, \ 1\leq i\leq n-1.$$

Consider
\[\begin{array}{lll}
d([y_1,e_2])&=[d(y_1),e_2]+[y_1,d(e_2)]=\\[1mm]
&=[a_1y_1+a_2y_2+\dots+a_{n-1}y_{y-1}, e_2]+ \\[1mm]
&+[y_1, b_1e_1+b_2e_2+a_2e_3+\dots+a_{n-2}e_{n-1}+b_ne_n]=\\[1mm]
&=b_1y_2+(a_1+b_2)\alpha_4y_3+(a_1\alpha_5+a_2\alpha_4+\alpha_5b_2)y_4+\\[1mm]
&+(a_1\alpha_{6}+a_2\alpha_{5}+a_3\alpha_{4}+\alpha_{6}b_2)y_5+ \dots +\\[1mm]
&+(a_1\theta+a_2\alpha_{n-1}+a_3\alpha_{n-2}+a_4\alpha_{n-3}+\dots+a_{n-3}\alpha_4+\theta b_2)y_{n-1}.\\[1mm]
\end{array}\]

On the other hand,
$$\begin{array}{lll}
d([y_1,e_2])&=\alpha_4d(y_3)+ \alpha_5d(y_4)+\dots+\alpha_{n-1}d(y_{n-2})+\theta d(y_{n-1})=\\[1mm]
&=5\alpha_4a_1y_3+(\alpha_4a_2+7\alpha_5 a_1)y_4+(\alpha_4a_3+\alpha_5a_2+9\alpha_6a_1)y_5+\dots+\\[1mm]
&+(\alpha_4a_{n-3}+\alpha_5a_{n-4}+\alpha_6a_{n-5}+\dots+\alpha_{n-1}a_2+(2n-3)\theta a_1)y_{n-1}.
\end{array}$$

Comparing the coefficients at the basis elements, we obtain that
$$ b_1=0,\quad (n-3)\theta a_1=0, \quad \alpha_i a_1=0, \ 4\leq i\leq n-1.$$

From $d([e_2,e_2])=[d(e_2),e_2]+[e_2,d(e_2)],$ we have $\alpha_n a_1=0.$
Verification of the property of derivation for the other products give the identity or already obtained restrictions.
\end{proof}

Similarly, to Proposition \ref{difL}, we have the description of the even derivations of the superalgebras for the other classes.

\begin{proposition}\label{difM}
An even derivation of  $M(\alpha_4, \alpha_5,\dots, \alpha_n, \theta, \tau)$ has the following form:
$$\left\{\begin{array}{lll}
d(e_1)=2a_1e_1+a_2e_3+a_3e_4+\dots+a_{n-1}e_n,\\[1mm]
d(e_2)=2a_1e_2+a_2e_3+a_3e_4+\dots+a_{n-2}e_{n-1}+b_ne_n,\\[1mm]
d(e_i)=2(i-1)a_1e_i+a_2e_{i+1}+a_{3}e_{i+2}+\dots+a_{n-i+1}e_n, & 3\leq i\leq n,\\[1mm]
d(y_i)=(2i-1)a_1y_i+a_2y_{i+1}+a_{3}y_{i+2}+\dots+a_{n-i+1}y_n, & 1\leq i\leq n,\\[1mm]
\end{array}\right.$$
where $\theta a_1=0, \quad \tau a_1=0, \quad \alpha_i a_1=0, \ 4\leq i\leq n.$
\end{proposition}

\begin{proof}
The proof is carried out using the property of derivation.
\end{proof}

\begin{proposition}\label{difH}
An even derivation of  $H(\beta_4, \beta_5,\dots, \beta_n, \delta, \gamma)$ has the following form:
$$\left\{\begin{array}{lll}
d(e_1)=2a_1e_1+a_2e_3+a_3e_4+\dots+a_{n-1}e_n,\\[1mm]
d(e_2)=b_2e_2,\\[1mm]
d(e_i)=2(i-1)a_1e_i+a_2e_{i+1}+a_{3}e_{i+2}+\dots+a_{n-i+1}e_n, & 3\leq i\leq n,\\[1mm]
d(y_i)=(2i-1)a_1y_i+a_2y_{i+1}+a_{3}y_{i+2}+\dots+a_{n-i+1}y_{n}, & 1\leq i\leq n,\\[1mm]
\end{array}\right.$$
where \begin{equation}\label{eq4.1}\begin{array}{lll}
\beta_i(2(i-2)a_1-b_2)=0, & 4\leq i\leq n,\\[1mm]
\delta(2(n-1)a_1-b_2)=0, & \gamma((n-1)a_1-b_2)=0.
\end{array}\end{equation}
\end{proposition}

\begin{proof}
The proof is carried out using the property of derivation.
\end{proof}

\begin{proposition}\label{difG}
An even derivation of  $G(\beta_4, \beta_5,\dots, \beta_n, \gamma)$ has the following form:
$$\left\{\begin{array}{lll}
d(e_1)=2a_1e_1+a_2e_3+a_3e_4+\dots+a_{n-1}e_n,\\[1mm]
d(e_2)=b_2e_2+b_ne_n,\\[1mm]
d(e_i)=2(i-1)a_1e_i+a_2e_{i+1}+a_{3}e_{i+2}+\dots+a_{n-i+1}e_n, & 3\leq i\leq n,\\[1mm]
d(y_i)=(2i-1)a_1y_i+a_2y_{i+1}+a_{3}y_{i+2}+\dots+a_{n-i}y_{n-1}, & 1\leq i\leq n-1,\\[1mm]
\end{array}\right.$$
where $\gamma((n-1)a_1-b_2)=0, \quad \beta_i(2(i-2)a_1-b_2)=0, \ 4\leq i\leq n.$
\end{proposition}

\begin{proof}
The proof is carried out using the property of derivation.
\end{proof}

From Proposition \ref{difL}, we have the following corollary.
\begin{corollary}\label{Cor1}
If $R$ is a non-nilpotent solvable Leibniz superalgebra with the nilradical from the class $L(\alpha_4, \alpha_5,\dots, \alpha_n, \theta),$ then
$\alpha_4=\alpha_5=\dots=\alpha_n=\theta=0.$
\end{corollary}

\begin{proof}
Suppose $\alpha_i\neq 0$ for some $i (4 \leq i \leq n)$ or $\theta \neq 0$.
Then from $$(n-3)\theta a_1=0,\quad \alpha_i a_1=0, \ 4\leq i\leq n,$$  we obtain that that $a_1=0,$ which implies the nilpotency of any even derivation of $L(\alpha_4, \alpha_5,\dots, \alpha_n, \theta).$ It is a contradiction to the non-nilpotency of $R$.
Therefore, we have $\alpha_4=\alpha_5=\dots=\alpha_n=\theta=0.$
\end{proof}

Thus, we conclude that solvable Leibniz superalgebra whose nilradical from the class $L(\alpha_4, \alpha_5,\dots, \alpha_n, \theta)$ exists only under the condition $\alpha_4=\alpha_5=\dots=\alpha_n=\theta=0$ and such solvable Leibniz superalgebras are classified in \cite{UzMat}.

\begin{theorem}\label{SL} Let $R$ be a solvable Leibniz superalgebra with nilradical $L(0, 0,\dots, 0, 0)$. Then it is isomorphic to the superalgebra
\[SL:\left\{\begin{array}{lll}
[e_1,e_1]=e_3, &[e_i,e_1]=e_{i+1},&    2 \le i \le n-1,\\[1mm]
 [y_j,e_1]=y_{j+1},&   & 1 \le j \le n-2,\\[1mm]
[e_1,y_1]= \frac12 y_2,&
[e_i,y_1]= \frac12 y_i,  &    2 \le i \le n-1,\\[1mm]
[y_1,y_1]=e_1,&[y_j,y_1]=e_{j+1},& 2 \le j \le n-1,\\[1mm]
[e_1,x]=2e_1,& [e_i,x]=2(i-1)e_i,& 2\leq i\leq n,\\[1mm]
[y_i,x]=(2i-1)y_i, & &1\leq i\leq n-1,\\[1mm]
[x,e_1]=-2e_1,&[x,y_1]=-y_1.\\[1mm]
\end{array}\right.\]
\end{theorem}

Analogously to the Corollary \ref{Cor1}, from Proposition \ref{difM}
for the solvable Leibniz superalgebras with the nilradical from the class $M(\alpha_4, \alpha_5,\dots, \alpha_n, \theta, \tau)$
we get
\begin{corollary}\label{Cor2}
If $L$ is a non-nilpotent solvable Leibniz superalgebra with the nilradical from the class  $M(\alpha_4, \alpha_5,\dots, \alpha_n, \theta, \tau),$ then
$\alpha_4=\alpha_5=\dots=\alpha_n=\theta=\tau=0.$
\end{corollary}

The following theorem describes a solvable Leibniz superalgebra whose nilradical is $M(0, 0,\dots, 0, 0, 0)$.  It is proved by a similar reason as in Theorem \ref{SL}.
\begin{theorem}\label{SM} Let $L$  be a solvable Leibniz superalgebra with nilradical $M(0, 0,\dots, 0, 0, 0)$. Then $L$ isomorphic to the superalgebra
\[SM:\left\{\begin{array}{lll}
[e_1,e_1]=e_3,&  [e_i,e_1]=e_{i+1},&     2 \le i \le n-1,\\[1mm]
[y_j,e_1]=y_{j+1}, & &   1 \le j \le n-1,\\[1mm]
[e_1,y_1]=  \frac12 y_2,&[e_i,y_1]=  \frac12 y_i, &    2 \le i \le n,\\[1mm]
[y_1,y_1]=e_1,& [y_j,y_1]=e_{j+1},& 2 \le j \le n-1,\\[1mm]
[e_1,x]=2e_1,& [e_i,x]=2(i-1)e_i,& 2\leq i\leq n,\\[1mm]
[y_i,x]=(2i-1)y_i,& & 1\leq i\leq n,\\[1mm]
[x,e_1]=-2e_1,&[x,y_1]=-y_1.\\[1mm]
\end{array}\right.\]
\end{theorem}

Now consider solvable Leibniz superalgebras whose nilradicals belong to the class $H(\beta_4, \beta_5, \dots, \beta_n,\delta, \gamma)$. Then from Proposition~\ref{difH}, we have the following result.

\begin{corollary}\label{corH}
If $L$ is a non-nilpotent solvable Leibniz superalgebra with nilradical from the class $H(\beta_4, \beta_5, \dots, \beta_n,\delta, \gamma)$, then:
$$(\beta_4, \beta_5, \dots, \beta_n,\delta, \gamma)=
\left\{\begin{array}{lll}
(0,0,\dots,0,0,0),\\[1mm]
(0,0,\dots,0, \beta_{t},0,\dots,0,0), & 4\leq t\leq n, & \beta_{t}\neq 0,\\[1mm]
(0,0,\dots,0,\delta,0), & & \delta\neq 0, \\[1mm]
(0,0,\dots,0, \beta_{\frac{n+3}2},0,\dots,0,\gamma), &  n \ \text{is odd}, &  \gamma\neq 0.
\end{array}\right.$$
\end{corollary}

\begin{proof}
By the conditions on the parameters of the $H(\beta_4, \beta_5, \dots, \beta_n,\delta, \gamma)$ from Proposition~\ref{difH}, we have the following cases:

\begin{itemize}
  \item If all parameters are equal to zero, we obtain a split superalgebra $H(0,0,\dots,0,0,0)$, which has non-nilpotent even derivation.
  \item If $\beta_i\neq 0, \beta_j\neq 0$ for some $i, j (4\leq i\neq j\leq n),$
  then from \eqref{eq4.1}, we have  $(a_1,b_2)=(0,0),$ which implies that all even derivations of the superalgebra are nilpotent.
    Therefore, in this case, there is no solvable Leibniz superalgebra with nilradical $H(\beta_4, \beta_5, \dots, \beta_n,\delta, \gamma)$.
  \item If $\beta_t\neq 0$ for some $t$ and $\beta_i = 0$ for $i\neq t,$ then $b_2=2(t-2)a_1$ and
  $$\delta(2(n-1)a_1-b_2)=0, \quad  \gamma((n-1)a_1-b_2)=0.$$
  From these equalities we have $\delta a_1(n+1-t)=0, \ \gamma a_1(n-2t +3)=0.$
 If $\delta \neq 0,$ then $a_1=0$ and the Leibniz superalgebra has only nilpotent even derivations. Thus $\delta = 0$ and
 \begin{itemize} \item if $\gamma = 0,$ then we have the superalgebras $H(0,0,\dots,0, \beta_{t},0,\dots,0,0),$ $4\leq t\leq n,$ $\beta_{t}\neq 0;$
 \item if $\gamma \neq 0,$ then in case of $t \neq \frac{n+3}2,$ we have that $a_1=0$ and the Leibniz superalgebra has only nilpotent even derivations which is contradiction with non-nilpotency of the Leibniz superalgebra $L.$ In case of $t = \frac{n+3}2$ we have the superalgebras $H(0,0,\dots,0, \beta_{\frac{n+3}2},0,\dots,0,\gamma).$
     Note that the case $t = \frac{n+3}2$ appears only for $n$ is odd.

 \end{itemize}
  \item If $\beta_i= 0$ for all $i (4\leq i\leq n)$ and $\delta\neq 0$, then $\gamma=0$ and we have the superalgebra $H(0,0, 0, \dots, 0, \delta, 0)$.
  \item If $\beta_i= 0$ for all $i (4\leq i\leq n),$ $\delta = 0$ and $\gamma\neq 0$, then we have the superalgebra $H(0,0, 0, \dots, 0, 0, \gamma).$
\end{itemize}
\end{proof}


Similarly, for the class of superalgebras $G(\beta_4, \beta_5, \dots, \beta_n, \gamma)$, we have

\begin{corollary} \label{corG}
If $L$ is a non-nilpotent solvable Leibniz superalgebra with nilradical from the class $G(\beta_4, \beta_5, \dots, \beta_n, \gamma)$, then:
$$(\beta_4, \beta_5, \dots, \beta_n, \gamma)=
\left\{\begin{array}{llll}
(0,0,\dots,0,0),\\[1mm]
(0,0,\dots,0, \beta_{t},0,\dots,0), & 4\leq t\leq n,\\[1mm]
(0,0,\dots,0, \beta_{\frac{n+3}2},0,\dots,0,\gamma), & n \ \text{is odd}, & \gamma\neq 0.
\end{array}\right.$$
\end{corollary}

Now using Corollary \ref{corH}, we classify solvable Leibniz superalgebras with nilradical $H(\beta_4, \beta_5, \dots, \beta_n,\delta, \gamma)$.

First we consider the case when the nilradical of solvable Leibniz superalgebra is the superalgebra $H(0, 0, \dots, 0,0).$
From Proposition \ref{difH}, it is easy to conclude that there are two nil-independent even derivations of the superalgebra $H(0, 0, \dots, 0,0)$
and other algebras have only one nil-independent even derivations.
Moreover, a superalgebra from the class $H(\beta_4, \beta_5, \dots, \beta_n,\delta, \gamma)$ is split if and only if all parameters are equal to zero, i.e., superalgebra isomorphic to
$H(0, 0, \dots, 0,0).$

\begin{theorem} \label{solvH(0)} Let $L=L_0\oplus L_1$ be a solvable Leibniz superalgebra whose nilradical is isomorphic to the superalgebra $H(0, 0, \dots, 0,0)$.
Then $L$ is isomorphic to the following pairwise non-isomorphic superalgebras:
$$MH_1:\left\{\begin{array}{lll}
[e_1,e_1]=e_3, & [e_i,e_1]=e_{i+1}, & 3\leq i\leq n-1,\\[1mm]
[y_j,e_1]=y_{j+1}, & & 1\leq j\leq n-1,\\[1mm]
[y_1,y_1]=e_1, &
[y_j,y_1]=e_{j+1}, &2\leq j\leq n-1,\\[1mm]
[e_1,y_1]=\frac{1}{2}y_2,&
[e_i,y_1]=\frac{1}{2}y_{i}, & 3\leq i\leq n,\\[1mm]
[y_i,x_1]=(2i-1)y_i, & & 1\leq i\leq n,\\[1mm]
[e_1,x_1]=2e_1,&
[e_i,x_1]=2(i-1)e_i, & 3\leq i\leq n,\\[1mm]
[x_1,e_1]=-2e_1, & [x_1,y_1]=-y_1, &
[e_2,x_2]=e_2,
\end{array}\right.
$$
$$
MH_2:\left\{\begin{array}{llll}
[e_1,e_1]=e_3, & [e_i,e_1]=e_{i+1}, & 3\leq i\leq n-1,\\[1mm]
[y_j,e_1]=y_{j+1}, & &1\leq j\leq n-1,\\[1mm]
[y_1,y_1]=e_1, &
[y_j,y_1]=e_{j+1}, &2\leq j\leq n-1,\\[1mm]
[e_1,y_1]=\frac{1}{2}y_2, &
[e_i,y_1]=\frac{1}{2}y_{i}, & 3\leq i\leq n,\\[1mm]
[y_i,x_1]=(2i-1)y_i, & & 1\leq i\leq n,\\[1mm]
[e_1,x_1]=2e_1, & [e_i,x_1]=2(i-1)e_i, & 3\leq i\leq n,\\[1mm]
[x_1,e_1]=-2e_1, &
[x_1,y_1]=-y_1, &
[e_2,x_2]=e_2, &
[x_2,e_2]=-e_2,
\end{array}\right.$$
$$H_1(b):\left\{\begin{array}{lllll}
[e_1,e_1]=e_3, &
[e_i,e_1]=e_{i+1}, & 3\leq i\leq n-1,\\[1mm]
[y_j,e_1]=y_{j+1}, && 1\leq j\leq n-1,\\[1mm]
[y_1,y_1]=e_1,& [y_j,y_1]=e_{j+1}, &2\leq j\leq n-1,\\[1mm]
[e_1,y_1]=\frac{1}{2}y_2, &
[e_i,y_1]=\frac{1}{2}y_{i}, & 3\leq i\leq n,\\[1mm]
[y_i,x]=(2i-1)y_i,& & 1\leq i\leq n,\\[1mm]
[e_1,x]=2e_1, &  [e_2,x]=be_2, & [e_i,x]=2(i-1)e_i,& 3\leq i\leq n, \\[1mm]
[x,e_1]=-2e_1, & [x,e_2]=-be_2, & [x,y_1]=-y_1, & b \neq0,
\end{array}\right.$$

$$H_2(b):\left\{\begin{array}{lllll}
[e_1,e_1]=e_3, & [e_i,e_1]=e_{i+1}, & 3\leq i\leq n-1,\\[1mm]
[y_j,e_1]=y_{j+1}, & & 1\leq j\leq n-1,\\[1mm]
[y_1,y_1]=e_1, & [y_j,y_1]=e_{j+1}, &2\leq j\leq n-1,\\[1mm]
[e_1,y_1]=\frac{1}{2}y_2, & [e_i,y_1]=\frac{1}{2}y_{i}, & 3\leq i\leq n,\\[1mm]
[y_i,x]=(2i-1)y_i,& & 1\leq i\leq n,\\[1mm]
[e_1,x]=2e_1,& [e_2,x]=be_2, & [e_i,x]=2(i-1)e_i, & 3\leq i\leq n,\\[1mm]
[x,e_1]=-2e_1, &  [x, y_1]=-y_1,
\end{array}\right.
$$

$$H_3:\left\{\begin{array}{lll}
[e_1,e_1]=e_3,& [e_i,e_1]=e_{i+1}, & 3\leq i\leq n-1,\\[1mm]
[y_j,e_1]=y_{j+1}, & &1\leq j\leq n-1,\\[1mm]
[y_1,y_1]=e_1, & [y_j,y_1]=e_{j+1}, &2\leq j\leq n-1,\\[1mm]
[e_1,y_1]=\frac{1}{2}y_2,& [e_i,y_1]=\frac{1}{2}y_{i}, & 3\leq i\leq n,\\[1mm]
[y_i,x]=(2i-1)y_i,& 1\leq i\leq n,\\[1mm]
[e_1,x]=2e_1,& [e_i,x]=2(i-1)e_i,& 3\leq i\leq n,\\[1mm]
[x,e_1]=-2e_1, & [x,y_1]=-y_1, & [x,x]=e_2,\\[1mm]
\end{array}\right.$$
$$H_4:\left\{\begin{array}{lllll}
[e_1,e_1]=e_3, & [e_i,e_1]=e_{i+1}, & 3\leq i\leq n-1,\\[1mm]
[y_j,e_1]=y_{j+1}, & & 1\leq j\leq n-1,\\[1mm]
[y_1,y_1]=e_1, &
[y_j,y_1]=e_{j+1}, &2\leq j\leq n-1,\\[1mm]
[e_1,y_1]=\frac{1}{2}y_2,&
[e_i,y_1]=\frac{1}{2}y_{i}, & 3\leq i\leq n,\\[1mm]
[e_1,x]=\sum\limits_{k=3}^{n}a_{k-1}e_k,&
[e_2,x]=e_2,&
[e_i,x]=\sum\limits_{k=i+1}^{n} a_{k+1-i}e_{k}, & 3\leq i\leq n,\\[1mm]
[y_i,x]=\sum\limits_{k=i+1}^{n} a_{k+1-i}y_{k}, & 1\leq i\leq n-1,
\end{array}\right.$$
$$H_5(\gamma):\left\{\begin{array}{lllll}
[e_1,e_1]=e_3, & [e_i,e_1]=e_{i+1}, & 3\leq i\leq n-1,\\[1mm]
[y_j,e_1]=y_{j+1}, & & 1\leq j\leq n-1,\\[1mm]
[y_1,y_1]=e_1, &
[y_j,y_1]=e_{j+1}, &2\leq j\leq n-1,\\[1mm]
[e_1,y_1]=\frac{1}{2}y_2,&
[e_i,y_1]=\frac{1}{2}y_{i}, & 3\leq i\leq n,\\[1mm]
[e_1,x]=\sum\limits_{k=3}^{n}a_{k-1}e_k,&
[e_2,x]=e_2,&
[e_i,x]=\sum\limits_{k=i+1}^{n} a_{k+1-i}e_{k}, & 3\leq i\leq n,\\[1mm]
[y_i,x]=\sum\limits_{k=i+1}^{n} a_{k+1-i}y_{k}, & 1\leq i\leq n,\\[1mm]
[x,e_2]=-e_2,&
[x,x]=\gamma e_2,& \gamma \in \{0; 1\}.\\[1mm]
\end{array}\right.$$
Note that, the first non-vanishing parameter $\{a_2, a_3,\dots,a_{n}\}$
in the algebras  $H_4$ and $H_5(\gamma)$ can be scaled to $1$.

\end{theorem}

\begin{proof}
From Proposition \ref{difH}, it is not difficult to see that the maximal number of nil-independent even derivations of the superalgebra
$N=H(0, 0, \dots, 0,0)$ is equal to $2.$ Thus, for the dimension of the solvable Leibniz superalgebras with nilradical $N,$ we have
$$\dim L - \dim N \leq 2. $$

\textbf{Case} $\dim L - \dim N = 2.$
Since the codimension of the nilradical $N$ is equal to $2$, we can choose a basis $\{e_1, e_2, \dots, e_n, x_1, x_2, y_1, y_2, \dots, y_m\}$ of $L$ such that $R_{x_1}$ and $R_{x_2}$ are nil-independent even derivations of $N.$
Then using Proposition \ref{difH}, we have that

$$\left\{\begin{array}{lll}
[e_1,x_1]=2e_1+a_2e_3+a_3e_4+\dots+a_{n-1}e_n,\\[1mm]
[e_i,x_1]=2(i-1)e_i+a_2e_{i+1}+a_{3}e_{i+2}+\dots+a_{n-i+1}e_n, & 3\leq i\leq n,\\[1mm]
[y_i,x_1]=(2i-1)y_i+a_2y_{i+1}+a_{3}y_{i+2}+\dots+a_{n-i+1}y_{n}, & 1\leq i\leq n,\\[1mm]
[e_1,x_2]=\alpha_2e_3+\alpha_3e_4+\dots+\alpha_{n-1}e_n,\\[1mm]
[e_2,x_2]=e_2,\\[1mm]
[e_i,x_2]=\alpha_2e_{i+1}+\alpha_{3}e_{i+2}+\dots+\alpha_{n-i+1}e_n, & 3\leq i\leq n-1,\\[1mm]
[y_i,x_2]=\alpha_2y_{i+1}+\alpha_{3}y_{i+2}+\dots+\alpha_{n-i+1}y_{n}, & 1\leq i\leq n-1.\\[1mm]
\end{array}\right.$$

Taking the following change of basis
$$\begin{cases}
y_i'=y_i+A_2y_{i+1}+A_3y_{i+2}+\dots+A_{n-i+1}y_n, & 1\leq i\leq n,\\[1mm]
e_1'=e_1+A_2e_2+A_3e_3+\dots+A_{n-1}e_n,\\[1mm]
e_2'=e_2,\\[1mm]
e_i'=e_i+A_2e_{i+1}+A_3e_{i+2}+\dots+A_{n-i+1}e_n,& 3\leq i\leq n,\\[1mm]
\end{cases}$$
 where
$$A_k=-\frac{a_k+A_2a_{k-1}+A_3a_{k-2}+\dots+A_{k-1}a_{2}}{2(k-1)}, \ 2\leq k\leq n,$$
we can assume $a_i=0, \ 2\leq i\leq n$.

From the multiplication of the nilradical and the properties of the right annihilator, we easily get $e_i, y_j \in \operatorname{Ann}_r(L)$ for $3\leq i\leq n$ and $2\leq j\leq n,$ i.e.,
$$[x_1,e_i]=[x_2,e_i]=0, \ 3\leq i\leq n, \qquad [x_1,y_j]=[x_2,y_j]=0, \ 2\leq j\leq n.$$

Put
 $$\begin{cases}
[x_1,y_1]=m_1y_1+m_2y_2+\dots+m_ny_n,\\[1mm]
[x_1,e_2]=p_1e_1+p_2e_2+\dots+p_ne_n,\\[1mm]
[x_2,y_1]=\gamma_1y_1+\gamma_2y_2+\dots+\gamma_ny_n,\\[1mm]
[x_2,e_2]=\delta_1e_1+\delta_2e_2+\dots+\delta_ne_n,\\[1mm]
[x_i,x_j]=c_{ij}^1e_1+c_{ij}^2e_2+\dots+c_{ij}^ne_n, & 1\leq i, j\leq 2.\\[1mm]
\end{cases}$$

Making the change $$x_1'=x_1-2m_2e_1-2\sum\limits_{k=3}^{n}m_ke_k, \quad x_2'=x_2-2\gamma_2e_1-2 \sum\limits_{k=3}^{n}\gamma_ke_k,$$ we can assume
$m_i=\gamma_i=0$ for $2\leq i \leq n$.

Considering the Leibniz superidentity for $\{x_1, y_1, y_1\}$ and $\{x_2, y_1, y_1\}$,  we derive
$$
[x_1, e_1]=2m_1e_1, \quad [x_2,e_1]=2\gamma_1e_1.$$
Similarly, if we apply the Leibniz superidentity on the triples
$\{x_1, e_2, e_1\}$, $\{x_2, e_2, e_1\}$, $\{e_1, x_1, e_1\}$, \ $\{e_1, x_2, e_1\}$, \ $\{x_1, x_1, e_1\}$, \ $\{x_2, x_2, e_1\}$, \ 
$\{x_1,y_1, x_1\}$, \ $\{x_2,y_1, x_2\}$, \ $\{e_1,x_1,x_2\}$, $\{e_1,x_2,x_1\}$, $\{x_1,x_2,e_1\}$, $\{x_2,x_1,e_1\}$, $\{x_1,x_2,y_1\}$ and $\{x_2, x_1, y_1\},$ we obtain
$$\begin{array}{llllll}
p_1=0, & p_i=0, & 3\leq i\leq n-1, \\[1mm]
\delta_1=0, & \delta_i=0, & 3\leq i\leq n-1, \\[1mm]
m_1=-1,&\gamma_1=0, & \alpha_i=0, &2\leq i\leq n, \\[1mm]
c_{ii}^1=0, & c_{ii}^k=0, & 1 \leq i \leq 2, & 3\leq k\leq n, \\[1mm]
c_{12}^1=0, &c_{12}^i=0, & 3\leq i\leq n,\\[1mm]
c_{21}^1=0, & c_{21}^i=0, & 3\leq i\leq n.  \\[1mm]
\end{array}$$

Moreover, the Leibniz superidentity for the triples $\{x_2,e_2,x_1\}$, $\{x_1, e_2,x_2\}$,
$\{x_2,x_1, e_2\}$, $\{x_1, x_2, e_2\}$ and $\{x_2, x_2, e_2\}$
gives us
$$\begin{array}{lllllllll}
\delta_n=0,&  p_n=0, &  p_2=0, & \delta_2(\delta_2+1)=0.
\end{array}$$

Changing the basis $x_1'=x_1-c_{12}^2e_2$
allows us to assume that $c_{12}^2=0$.

Now using the Leibniz superidentities for $\{x_1,x_1,x_2\},$
$\{x_2, x_1, x_2\}$, $\{x_2, x_2, x_2\},$
we have $c_{11}^2=0,$ $c_{21}^2=0,$ $\delta_2c_{22}^2=0.$

\begin{itemize}
\item If $\delta_2=0$, then changing $x_2'=x_2-c_{22}^2  e_2,$ we can assume $c_{22}^2=0$ and obtain the superalgebra $MH_1.$

\item If $\delta_2=-1$, then $c_{22}^2=0$ and obtain the superalgebra $MH_2.$
\end{itemize}

\textbf{Case} $\dim L - \dim N = 1.$
Since, the operator of right multiplication $R_{x}$ is a derivation of $H(0, 0, \dots, 0),$ then
using Proposition \ref{difH}, we can assume that

$$\left\{\begin{array}{lll}
[e_1,x]=2a_1e_1+a_2e_3+a_3e_4+\dots+a_{n-1}e_n,\\[1mm]
[e_2,x]=b_2e_2,\\[1mm]
[e_i,x]=2(i-1)a_1e_i+a_2e_{i+1}+a_{3}e_{i+2}+\dots+a_{n-i+1}e_n, & 3\leq i\leq n,\\[1mm]
[y_i,x]=(2i-1)a_1y_i+a_2y_{i+1}+a_{3}y_{i+2}+\dots+a_{n-i+1}y_{n}, & 1\leq i\leq n.\\[1mm]
\end{array}\right.$$

Since $(a_1,b_2)\neq (0,0)$, we divide this case into two subcases:

\textbf{Subcase 1.} Let $a_1\neq 0,$ then we may suppose $a_1 =1.$
Taking the change of basis
$$\begin{cases}
y_i'=y_i+A_2y_{i+1}+A_3y_{i+2}+\dots+A_{n-i+1}y_n, & 1\leq i\leq n,\\[1mm]
e_1'=e_1+A_2e_2+A_3e_3+\dots+A_{n-1}e_n,\\[1mm]
e_2'=e_2,\\[1mm]
e_i'=e_i+A_2e_{i+1}+A_3e_{i+2}+\dots+A_{n-i+1}e_n,& 3\leq i\leq n,\\[1mm]
\end{cases}$$
 where
$$A_k=-\frac{a_k+A_2a_{k-1}+A_3a_{k-2}+\dots+A_{k-1}a_{2}}{2(k-1)}, \ 2\leq k\leq n,$$
 we can assume $a_i=0, \ 2\leq i\leq n.$

Since $e_i, y_j \in \operatorname{Ann}_r(L)$ for $3\leq i\leq n$ and $2\leq j\leq n,$  we conclude
$$[x,e_i]=0, \ 3\leq i\leq n, \quad [x,y_j]=0, \ 2\leq j\leq n.$$

Put
 $$\begin{cases}
[x,y_1]=\alpha_1y_1+\alpha_2y_2+\dots+\alpha_ny_n,\\[1mm]
[x,e_2]=\mu_1e_1+\mu_2e_2+\dots+\mu_ne_n,\\[1mm]
[x,x]=\gamma_1e_1+\gamma_2e_2+\dots+\gamma_ne_n.\\[1mm]
\end{cases}$$

Using Leibniz superidentity we get:
$$[x,e_1]=[x,[y_1,y_1]]=2[[x,y_1], y_1]=2\alpha_1e_1+2\alpha_2e_3+\dots+2\alpha_{n-1}e_{n}.$$
Taking the change $x'=x-2(\alpha_3e_3+\dots+\alpha_{n}e_{n}),$ one can assume $\alpha_i=0$ for $3\leq i\leq n.$
From $0=[y_1,[x,x]]$ and $[x,[y_1,x]]=[[x,y_1],x]-[[x,x],y_1],$
we obtain
$$\gamma_1=0, \quad  \alpha_2=0, \quad \gamma_i=0, \quad 3\leq i\leq n,$$

Since $[e_1,x]+[x,e_1] = 2(\alpha_1+1)e_1\in \operatorname{Ann}_r(L)$, we get $\alpha_1=-1.$

Considering the Leibniz superidentity for $\{x,e_2,e_1\}$, $\{x,e_2,y_1\}$, $\{x,x,e_2\}$, $\{x,x,x\}$, we derive the following restrictions:
$$\begin{array}{llll}
\mu_1=0, & \mu_i=0, & 3\leq i\leq n, \\[1mm]
 \mu_2(\mu_2+b_2)=0, & \gamma_2\mu_2=0.
\end{array}$$

\begin{itemize}
\item Let $\mu_2\neq 0$, then $\gamma_2=0,$ $b_2=-\mu_2,$ and obtain the superalgebra $H_1(b).$

\item Let $\mu_2= 0$,
\begin{itemize}
\item If $b_2\neq 0$, then taking $x'=x- \frac{\gamma_2}{b_2}e_2,$  we can suppose $\gamma_2 = 0,$ and obtain the superalgebra $H_2(b)$ for $b\neq 0.$

\item If $b_2= 0$, then in case of $\gamma_2=0,$ we have the superalgebra $H_2(b)$ for $b = 0,$
in case of $\gamma_2\neq 0$ making the change $e_2' = \gamma_2 e_2$, we obtain the superalgebra $H_3$.
\end{itemize}
\end{itemize}

\textbf{Subcase 2.} $a_1=0$, then $b_2\neq 0$ and we may suppose $b_2= 1.$ Then
$$\left\{\begin{array}{lll}
[e_1,x]=a_2e_3+a_3e_4+\dots+a_{n-1}e_n,\\[1mm]
[e_2,x]=e_2,\\[1mm]
[e_i,x]=a_2e_{i+1}+a_{3}e_{i+2}+\dots+a_{n-i+1}e_n, & 3\leq i\leq n,\\[1mm]
[y_i,x]=a_2y_{i+1}+a_{3}y_{i+2}+\dots+a_{n-i+1}y_{n}, & 1\leq i\leq n.\\[1mm]
\end{array}\right.$$

Put
 $$\begin{cases}
[x,y_1]=\alpha_1y_1+\alpha_2y_2+\dots+\alpha_ny_n,\\[1mm]
[x,e_2]=\mu_1e_1+\mu_2e_2+\dots+\mu_ne_n,\\[1mm]
[x,x]=\gamma_1e_1+\gamma_2e_2+\dots+\gamma_ne_n.\\[1mm]
\end{cases}$$

Using Leibniz superidentity for triples $\{x,y_1,y_1\}$, we get
$$[x,e_1]=2\alpha_1e_1+2\alpha_2e_3+\dots+2\alpha_{n-1}e_{n}.$$
By changing the basis $x'=x-2(\alpha_2e_1+\alpha_3e_3+\dots+\alpha_{n}e_{n})$, we can assume $\alpha_i=0$ for $2\leq i\leq n.$

Considering the Leibniz superidentity for elements  $\{y_1,x,x\}$, $\{x,e_2,e_1\}$, $\{x,e_2,y_1\}$, $\{e_1,x,y_1\}$, $\{x,x,e_2\}$, $\{x,y_1,x\}$, we obtain the following restrictions:
$$\begin{cases}
\gamma_1=0, \quad \mu_1=0, \quad \mu_i=0, \quad 3\leq i\leq n,\quad
\alpha_1=0,\\[1mm]
\mu_2(\mu_2+1)=0, \quad \gamma_i=0, \quad 3\leq i\leq n.
\end{cases}$$

\textbf{Subcase 2.1.} If $\mu_2=0$, then by $x'=x-\gamma_2e_2$, we may suppose that $\gamma_2=0$ and obtain the superalgebra $H_4.$

\textbf{Subcase 2.2.} If $\mu_2=-1$, then we have the superalgebra $H_5(\gamma).$
\end{proof}

Now we consider the case when the nilradical is a non-split superalgebra from the class $H(\beta_4, \beta_5, \dots, \beta_n,\delta, \gamma),$ i.e., at least one of the parameters is not equal to zero.

\begin{theorem} \label{solvH} Let $L=L_0\oplus L_1$ be a solvable Leibniz superalgebra whose nilradical is isomorphic to a non-split superalgebra $N$ from the class $H(\beta_4, \beta_5, \dots, \beta_n,\delta, \gamma)$. Then $L$ is isomorphic to one of the following pairwise non-isomorphic superalgebras:
$$\small SH_1({t})(4\leq t\leq n):\left\{\begin{array}{lll}
[e_1,e_1]=e_3,& [e_i,e_1]=e_{i+1}, & 3\leq i\leq n-1,\\[1mm]
[y_j,e_1]=y_{j+1},& &1\leq j\leq n-1,\\[1mm]
[y_1,y_1]=e_1,& [y_j,y_1]=e_{j+1}, &2\leq j\leq n-1,\\[1mm]
[e_1,y_1]=\frac{1}{2}y_2,& [e_i,y_1]=\frac{1}{2}y_{i}, & 3\leq i\leq n,\\[1mm]
[e_1,e_2]=e_{t},& [e_j,e_2]=e_{j+t-2},&3\leq j\leq n-2,\\[1mm]
[y_1,e_2]=y_{t-1},& [y_j,e_2]=y_{j+t-2},&2\leq j\leq n-2,\\[1mm]
[y_i,x]=(2i-1)y_i,& & 1\leq i\leq n,\\[1mm]
 [x,y_1]=-y_1,\\[1mm]
[e_1,x]=2e_1,& [x,e_1]=-2e_1, \\[1mm]
[e_2,x]=2(t-2)e_2,& [x,e_2]=-2(t-2)e_2-2e_{t-1},\\[1mm]
[e_i,x]=2(i-1)e_i,& & 3\leq i\leq n, \\[1mm]
\end{array}\right.$$
$$SH_2:\left\{\begin{array}{lll}
[e_1,e_1]=e_3,& [e_i,e_1]=e_{i+1}, & 3\leq i\leq n-1,\\[1mm]
[y_j,e_1]=y_{j+1},& &1\leq j\leq n-1,\\[1mm]
[y_1,y_1]=e_1,& [y_j,y_1]=e_{j+1}, &2\leq j\leq n-1,\\[1mm]
[e_1,y_1]=\frac{1}{2}y_2,& [e_i,y_1]=\frac{1}{2}y_{i}, & 3\leq i\leq n,\\[1mm]
[y_1,e_2]=y_n,\\[1mm]
[y_i,x]=(2i-1)y_i,& & 1\leq i\leq n,\\[1mm]
 [x,y_1]=-y_1,\\[1mm]
[e_1,x]=2e_1,& [x,e_1]=-2e_1, \\[1mm]
[e_2,x]=2(n-1)e_2,& [x,e_2]=-2(n-1)e_2 -2e_n,\\[1mm]
[e_i,x]=2(i-1)e_i,& & 3\leq i\leq n. \\[1mm]
\end{array}\right.$$
$$SH_3(\gamma) (n \ \text{is odd}):\left\{\begin{array}{lll}
[e_1,e_1]=e_3,&[e_i,e_1]=e_{i+1}, & 3\leq i\leq n-1,\\[1mm]
[y_j,e_1]=y_{j+1}, & &1\leq j\leq n-1,\\[1mm]
[y_1,y_1]=e_1,& [y_j,y_1]=e_{j+1}, &2\leq j\leq n-1,\\[1mm]
[e_1,y_1]=\frac{1}{2}y_2,& [e_i,y_1]=\frac{1}{2}y_{i}, & 3\leq i\leq n,\\[1mm]
[e_1,e_2]=e_{\frac{n+3}2},& [e_j,e_2]=e_{j+\frac{n-1}2},&3\leq j\leq n-2,\\[1mm]
[y_1,e_2]=y_{\frac{n+1}2},& [y_j,e_2]=y_{j+\frac{n-1}2},&2\leq j\leq n-2,\\[1mm]
[e_2,e_2]=\gamma e_n,& & \gamma\neq 0,\\[1mm]
[y_i,x]=(2i-1)y_i,& & 1\leq i\leq n,\\[1mm]
 [x,y_1]=-y_1,\\[1mm]
[e_1,x]=2e_1,& [x,e_1]=-2e_1, \\[1mm]
[e_2,x]=(n-1)e_2,& [x,e_2]=-(n-1)e_2-2e_\frac{n+1}2,\\[1mm]
[e_i,x]=2(i-1)e_i,& & 3\leq i\leq n, \\[1mm]
\end{array}\right.$$
$$SH_4:\left\{\begin{array}{lll}
[e_1,e_1]=e_3,& [e_i,e_1]=e_{i+1}, & 3\leq i\leq n-1,\\[1mm]
[y_j,e_1]=y_{j+1},& &1\leq j\leq n-1,\\[1mm]
[y_1,y_1]=e_1,& [y_j,y_1]=e_{j+1}, &2\leq j\leq n-1,\\[1mm]
[e_1,y_1]=\frac{1}{2}y_2,& [e_i,y_1]=\frac{1}{2}y_{i}, & 3\leq i\leq n,\\[1mm]
[e_2,e_2]=e_n,\\[1mm]
[y_i,x]=(2i-1)y_i,& & 1\leq i\leq n,\\[1mm]
 [x,y_1]=-y_1,\\[1mm]
[e_1,x]=2e_1,& [x,e_1]=-2e_1, \\[1mm]
[e_2,x]=(n-1)e_2,& [x,e_2]=-(n-1)e_2,\\[1mm]
[e_i,x]=2(i-1)e_i,& & 3\leq i\leq n, \\[1mm]
\end{array}\right.$$
\end{theorem}

\begin{proof}
From Proposition \ref{difH}, we have that any non-split superalgebra from the class $H(\beta_4, \beta_5, \dots, \beta_n,\delta, \gamma)$
has a maximum one nil-independent even derivation. Then we have that, the codimension of the sol\-vable superalgebra $L$ with such nilradicals may be equal to one.
Let $\{e_1,e_2,\dots, e_n, , x,  y_1,y_2\ \dots, y_n\}$ be a basis of the superalgebra $L = L_0 \oplus L_1,$ such that
$L_0 = \{e_1,e_2,\dots, e_n, x\}$ and $L_1 = \{y_1,y_2\ \dots, y_n\}.$ Since, the operator of right multiplication $R_{x}$ is a derivation of $H(\beta_4, \beta_5, \dots, \beta_n,\delta, \gamma),$ then
using Proposition \ref{difH}, we can assume that
$$\begin{array}{lll}
[y_i,x]=(2i-1)a_1y_i+a_2y_{i+1}+\dots+a_{n-i+1}y_n,& 1\leq i\leq n,\\[1mm]
[e_1,x]=2a_1e_1+a_2e_{3}+\dots+a_{n-1}e_n,\\[1mm]
[e_2,x]=b_2e_2,\\[1mm]
[e_i,x]=2(i-1)a_1 e_i+a_2e_{i+1}+\dots+a_{n-i+1}e_n,& 3\leq i\leq n.
\end{array}$$

Moreover, if $\beta_i \neq 0,$ for some $i (4 \leq i \leq n),$ then $b_2=2(i-2)a_1,$
if $\beta_i = 0,$ for any $i (4 \leq i \leq n)$ and $\delta \neq 0,$ then $b_2=2(n-1)a_1,$
if $\beta_i = 0,$ for any $i (4 \leq i \leq n),$ $\delta = 0$ and $\gamma \neq 0,$ then $b_2=(n-1)a_1.$

Thus, from non-nilpotency of even derivation of the
non-split superalgebra from the class $H(\beta_4, \beta_5, \dots, \beta_n,\delta, \gamma),$ we conclude that $a_1 \neq 0.$ Therefore, we can suppose $a_1=1$ and considering the following change of basis
$$\left\{\begin{array}{llll}
y_i'=y_i+A_2y_{i+1}+A_3y_{i+2}+\dots+A_{n-i+1}y_n,& 1\leq i\leq n,\\[1mm]
e_1'=e_1+A_2e_3+A_3e_4+\dots+A_{n-1}e_n,\\[1mm]
e_2=e_2,\\[1mm]
e_i'=e_i+A_2e_{i+1}+A_3e_{i+2}+\dots+A_{n-i+1}e_n, &3\leq i\leq n,
\end{array}\right.$$
 where
$$A_k=-\frac{a_k+A_2a_{k-1}+A_3a_{k-2}+\dots+A_{k-1}a_{2}}{2(k-1)}, \ 2\leq k\leq n,$$
 one can assume $a_i=0, \ 2\leq i\leq n.$

From \eqref{eqH}, we can easily get that the basis elements $e_3, e_4, \dots, e_n, y_2, y_3, \dots, y_n$ belongs to the right annihilator of the superalgebra $L.$ Thus, we have
$$[x,e_i]=0, \ 3\leq i\leq n, \qquad [x,y_j]=0, \ 2\leq j\leq n.$$

Put
 $$\begin{cases}
[x,y_1]=\alpha_1y_1+\alpha_2y_2+\dots+\alpha_ny_n,\\[1mm]
[x,e_2]=\mu_1e_1+\mu_2e_2+\dots+\mu_ne_n,\\[1mm]
[x,x]=\gamma_1e_1+\gamma_2e_2+\dots+\gamma_ne_n.\\[1mm]
\end{cases}$$

Using Leibniz superidentity we get:
$$[x,e_1]=[x,[y_1,y_1]]=2[[x,y_1], y_1]=2\alpha_1e_1+2\alpha_2e_3+\dots+2\alpha_{n-1}e_{n}.$$

Making the change  $x'=x-2(\alpha_2e_1 + \alpha_3e_3+\dots+\alpha_{n}e_{n}),$ we can assume that $\alpha_i=0, \ 2\leq i\leq n.$

Moreover, from $0=[y_1,[x,x]],$ we obtain
\begin{equation}\label{eq_1} \gamma_1=0, \quad \gamma_2\beta_i=0, \quad 4\leq i\leq n,  \quad \gamma_2 \delta=0. \end{equation}

The Leibniz superidentity $[x,[y_1,x]]=[[x,y_1],x]-[[x,x],y_1]$ gives us the following
$$\gamma_i=0, \quad 3\leq i\leq n.$$

Since $[e_1,x]+[x,e_1] \in \operatorname{Ann}_r(L),$ we get that $\alpha_1=-1.$
Thus, we have the following multiplications
$$\left\{\begin{array}{llll}
[y_i,x]=(2i-1)y_i,& 1\leq i\leq n,\\[1mm]
[e_1,x]=2e_1, & [e_2,x]=b_2e_2, &
[e_i,x]=2(i-1)e_i,& 3\leq i\leq n,\\[1mm]
[x,e_1]=-2e_1, & [x,e_2]=\sum\limits_{k=1}^n\mu_ke_k,\\[1mm]
[x,y_1]=-y_1, & [x,x]=\gamma_2e_2.\\[1mm]
\end{array}\right.$$

Now, using Corollary \ref{corH}, consider the following cases:

\textbf{Case 1.} Let $N=H(0,0, \dots, \beta_{t},0, \dots, 0, 0),$ where $4\leq t\leq n,$ i.e., nilradical $N$ has the multiplication
\begin{equation}\label{eq(Case1)}\left\{\begin{array}{lll}
[e_1,e_1]=e_3,&[e_i,e_1]=e_{i+1}, & 3\leq i\leq n-1,\\[1mm]
[y_j,e_1]=y_{j+1},& &1\leq j\leq n-1,\\[1mm]
[y_1,y_1]=e_1,&[y_j,y_1]=e_{j+1}, &2\leq j\leq n,\\[1mm]
[e_1,y_1]=\frac{1}{2}y_2,& [e_i,y_1]=\frac{1}{2}y_{i}, & 3\leq i\leq n-1,\\[1mm]
[e_1,e_2]=\beta_te_{t},&[e_j,e_2]=\beta_{t}e_{j+t-2},&3\leq j\leq n-2,\\[1mm]
[y_1,e_2]=\beta_{t}y_{t-1},&[y_j,e_2]=\beta_{t}y_{j+t-2},&2\leq j\leq n-2.
\end{array}\right.\end{equation}

Since $\beta_t\neq 0,$ then from \eqref{eq_1}, we get $\gamma_2=0$ and changing $e_2'=\frac{1}{\beta_t}e_2,$ we have $\beta_t=1$ and $b_2=2(t-2).$
Applying the Leibniz superidentity for the triples $\{e_1, x, e_2\}, \{x, e_2, y_1\},$ we have
$$\mu_1=0, \ \mu_2=2(2-t), \quad  \mu_i=0, \ 3\leq i\neq t-1\leq n, \qquad \mu_{t-1}=-2.$$

Therefore, we obtain the superalgebra $SH_1(t).$

\textbf{Case 2.} Consider the case when nilradical is
$$H(0,0, \dots, 0, \delta,0):\left\{\begin{array}{lll}
[e_1,e_1]=e_3,& [e_i,e_1]=e_{i+1}, & 3\leq i\leq n-1,\\[1mm]
[y_j,e_1]=y_{j+1},& &1\leq j\leq n-1,\\[1mm]
[y_1,y_1]=e_1,&[y_j,y_1]=e_{j+1}, &2\leq j\leq n-1,\\[1mm]
[e_1,y_1]=\frac{1}{2}y_2,&[e_i,y_1]=\frac{1}{2}y_{i}, & 3\leq i\leq n,\\[1mm]
[y_1,e_2]=\delta y_n.
\end{array}\right.$$

Since $\delta\neq 0,$ then from \eqref{eq_1} we have $\gamma_2=0$ and
by changing $e_2'=\frac{1}{\delta} e_2,$ we can assume $\delta=1.$

From the Leibniz superidentity for $\{x,e_2, y_1\},$ $\{y_1, x, e_2\},$  we obtain
$$\mu_1=0,\quad \mu_2=2(1-n),\quad \mu_i=0, \quad 3\leq i\leq n-1, \quad \mu_n=-2.$$

Therefore, we have the superalgebra $SH_2$.

\textbf{Case 3.} Let the nilradical of the superalgebra is
 $$\footnotesize H(0,0, \dots, \beta_{\frac{n+3}2},0, \dots, 0, \gamma): \left\{\begin{array}{lll}
[e_1,e_1]=e_3,&[e_i,e_1]=e_{i+1}, & 3\leq i\leq n-1,\\[1mm]
[y_j,e_1]=y_{j+1}, &1\leq j\leq n-1,\\[1mm]
[y_1,y_1]=e_1,& [y_j,y_1]=e_{j+1}, &2\leq j\leq n-1,\\[1mm]
[e_1,y_1]=\frac{1}{2}y_2,&[e_i,y_1]=\frac{1}{2}y_{i}, & 3\leq i\leq n,\\[1mm]
[e_1,e_2]=\beta_{\frac{n+3}2}e_{\frac{n+3}2},&[e_j,e_2]=\beta_{\frac{n+3}2}e_{j+\frac{n-1}2},&3\leq j\leq n-2,\\[1mm]
[y_1,e_2]=\beta_{\frac{n+3}2}y_{\frac{n+1}2},&[y_j,e_2]=\beta_{\frac{n+3}2}y_{j+\frac{n-1}2},&2\leq j\leq n-2,\\[1mm]
[e_2,e_2]=\gamma e_n.
\end{array}\right.$$

In this case, we have $b_2=n-1.$

If $\beta_{\frac{n+3}2} \neq 0,$ then from \eqref{eq_1}, we get $\gamma_2=0.$
Taking the change $e_2'=\frac{1}{\beta_{\frac{n+3}2}}e_2$, we can suppose $\beta_{\frac{n+3}2}=1.$
Considering the Leibniz superidentity for $\{x, e_2, y_1\}$, $\{e_1, x, e_2\},$ we get
$$\mu_1=0, \quad \mu_2=1-n,\quad \mu_\frac{n+1}{2}=-2, \quad \mu_i=0, \quad 3\leq i \leq n,  (i\neq \frac{n+1}{2}).$$

Thus, we have the superalgebra $SH_3(\gamma).$

If $\beta_{\frac{n+3}2} = 0,$ then $\gamma\neq 0$ and by changing $e_2'=\frac{1}{\sqrt{\gamma}}e_2,$ we can suppose $\gamma=1.$
Considering Leibniz superidentity for $\{x, e_2, y_1\}$, $\{e_2, x, x\}$, $\{e_2,x,e_2\},$ we get
$$\gamma_2=0, \quad \mu_1=0,\quad \mu_2=1-n,\quad \mu_i=0,\quad 3\leq i\leq n.$$
Thus, we obtain the superalgebra $SH_4$.
\end{proof}

Now we give the description of solvable Leibniz superalgebras whose nilradical is isomorphic to the superalgebra from the class $G(\beta_4, \beta_5, \dots, \beta_n, \gamma)$.
In the following theorem, we consider the case when nilradical is $G(0,0,\dots,0)$.

\begin{theorem} Let $L=L_0\oplus L_1$ be a solvable Leibniz superalgebra whose nilradical is isomorphic to the superalgebra $G(0,0,\dots,0)$. Then $L$ is isomorphic to one of the following pairwise non-isomorphic superalgebras:
$$MG_1:\left\{\begin{array}{llll}
[e_1,e_1]=e_3,&[e_i,e_1]=e_{i+1}, & 3\leq i\leq n-1,\\[1mm]
[y_j,e_1]=y_{j+1},& &1\leq j\leq n-2,\\[1mm]
[y_1,y_1]=e_1,& [y_j,y_1]=e_{j+1}, &2\leq j\leq n-1,\\[1mm]
[e_1,y_1]=\frac{1}{2}y_2,& [e_i,y_1]=\frac{1}{2}y_{i}, & 3\leq i\leq n-1,\\[1mm]
[e_1,x_1]=2e_1,&
[e_i,x_1]=2(i-1)e_i, & 3\leq i\leq n,\\[1mm]
[y_i,x_1]=(2i-1)y_i,& & 1\leq i\leq n-1,\\[1mm]
[x_1,e_1]=-2e_1,& [x_1,y_1]=-y_1, & [e_2,x_2]=e_2, \\[1mm]
\end{array}\right.$$
$$MG_2:\left\{\begin{array}{llll}
[e_1,e_1]=e_3,&[e_i,e_1]=e_{i+1}, & 3\leq i\leq n-1,\\[1mm]
[y_j,e_1]=y_{j+1},& &1\leq j\leq n-2,\\[1mm]
[y_1,y_1]=e_1,& [y_j,y_1]=e_{j+1}, &2\leq j\leq n-1,\\[1mm]
[e_1,y_1]=\frac{1}{2}y_2,& [e_i,y_1]=\frac{1}{2}y_{i}, & 3\leq i\leq n-1,\\[1mm]
[e_1,x_1]=2e_1,&[e_i,x_1]=2(i-1)e_i, & 3\leq i\leq n,\\[1mm]
[y_i,x_1]=(2i-1)y_i, & 1\leq i\leq n-1,\\[1mm]
[x_1,e_1]=-2e_1, & [x_1,y_1]=-y_1, \\[1mm]
[e_2,x_2]=e_2, & [x_2,e_2]=-e_2, \\[1mm]
\end{array}\right.$$

$$G_1(b):\left\{\begin{array}{llll}
[e_1,e_1]=e_3,&[e_i,e_1]=e_{i+1}, & 3\leq i\leq n-1,\\[1mm]
[y_j,e_1]=y_{j+1},& &1\leq j\leq n-2,\\[1mm]
[y_1,y_1]=e_1,& [y_j,y_1]=e_{j+1}, &2\leq j\leq n-1,\\[1mm]
[e_1,y_1]=\frac{1}{2}y_2,& [e_i,y_1]=\frac{1}{2}y_{i}, & 3\leq i\leq n-1,\\[1mm]
[y_i,x]=(2i-1)y_i, & & 1\leq i\leq n-1,\\[1mm]
[e_1,x]=2e_1,& [e_2,x]=be_2,& [e_i,x]=2(i-1)e_i,& 3\leq i\leq n,\\[1mm]
 [x,e_1]=-2e_1, & [x,e_2]=-be_2,& [x,y_1]=-y_1, & b\neq 0,
\end{array}\right.$$
$$G_2(b):\left\{\begin{array}{llllll}
[e_1,e_1]=e_3,&[e_i,e_1]=e_{i+1}, & 3\leq i\leq n-1,\\[1mm]
[y_j,e_1]=y_{j+1},& &1\leq j\leq n-2,\\[1mm]
[y_1,y_1]=e_1,& [y_j,y_1]=e_{j+1}, &2\leq j\leq n-1,\\[1mm]
[e_1,y_1]=\frac{1}{2}y_2,& [e_i,y_1]=\frac{1}{2}y_{i}, & 3\leq i\leq n-1,\\[1mm]
[y_i,x]=(2i-1)y_i, & & 1\leq i\leq n-1,\\[1mm]
[e_1,x]=2e_1, & [e_2,x]=be_2, & [e_i,x]=2(i-1)e_i,& 3\leq i\leq n,\\[1mm]
[x,e_1]= - 2e_1,& [x,y_1]=-y_1,\\[1mm]
\end{array}\right.$$

$$G_3:\left\{\begin{array}{lll}
[e_1,e_1]=e_3, & [e_i,e_1]=e_{i+1}, & 3\leq i\leq n-1,\\[1mm]
[y_j,e_1]=y_{j+1}, &&1\leq j\leq n-2,\\[1mm]
[y_1,y_1]=e_1, & [y_j,y_1]=e_{j+1}, &2\leq j\leq n-1,\\[1mm]
[e_1,y_1]=\frac{1}{2}y_2, & 
[e_i,y_1]=\frac{1}{2}y_{i}, & 3\leq i\leq n-1,\\[1mm]
[y_i,x]=(2i-1)y_i,& &1\leq i\leq n-1,\\[1mm]
[e_1,x]=2e_1, & [x,e_1]=-2e_1,\\[1mm]
[e_2,x]=2(n-1)e_2+e_n, & 
[e_i,x]=2(i-1)e_i,& 3\leq i\leq n,\\[1mm]
[x,y_1]=-y_1,\\[1mm]
\end{array}\right.$$ $$
G_4(\gamma,b):\left\{\begin{array}{lll}
[e_1,e_1]=e_3, &
[e_i,e_1]=e_{i+1}, & 3\leq i\leq n-1,\\[1mm]
[y_j,e_1]=y_{j+1}, & &1\leq j\leq n-2,\\[1mm]
[y_1,y_1]=e_1, &
[y_j,y_1]=e_{j+1}, &2\leq j\leq n-1,\\[1mm]
[e_1,y_1]=\frac{1}{2}y_2, & [e_i,y_1]=\frac{1}{2}y_{i}, & 3\leq i\leq n-1,\\[1mm]
[y_i,x]=(2i-1)y_i,& 1\leq i\leq n-1,\\[1mm]
[e_1,x]=2e_1, & [x,e_1]=-2e_1,\\[1mm]
[e_2,x]=be_n, &
[e_i,x]=2(i-1)e_i,& 3\leq i\leq n,\\[1mm]
[x,y_1]=-y_1,& [x,x]=\gamma e_2, & (\gamma, b)=(0,1), (1,0), (1,1)
\end{array}\right.$$
$$G_5:\left\{\begin{array}{llllll}
[e_1,e_1]=e_3,& [e_i,e_1]=e_{i+1}, & 3\leq i\leq n-1,\\[1mm]
[y_j,e_1]=y_{j+1},& &1\leq j\leq n-2,\\[1mm]
[y_1,y_1]=e_1,& [y_j,y_1]=e_{j+1}, &2\leq j\leq n-1,\\[1mm]
[e_1,y_1]=\frac{1}{2}y_2,& [e_i,y_1]=\frac{1}{2}y_{i}, & 3\leq i\leq n-1,\\[1mm]
[e_1,x]=\sum\limits_{k=3}^na_{k-1}e_k,&
[e_2,x]=e_2,&
[e_i,x]=\sum\limits_{k=i+1}^na_{k+1-i}e_k, & 3\leq i\leq n,\\[1mm]
[x,x]=\gamma e_n,& [y_i,x]=\sum\limits_{k=i+1}^{n-1}a_{k+1-i}y, & 1\leq i\leq n-1,
\end{array}\right.$$
$$G_6:\left\{\begin{array}{llll}
[e_1,e_1]=e_3,&
[e_i,e_1]=e_{i+1}, & 3\leq i\leq n-1,\\[1mm]
[y_j,e_1]=y_{j+1},& &1\leq j\leq n-2,\\[1mm]
[y_1,y_1]=e_1,& [y_j,y_1]=e_{j+1}, &2\leq j\leq n-1,\\[1mm]
[e_1,y_1]=\frac{1}{2}y_2,& [e_i,y_1]=\frac{1}{2}y_{i}, & 3\leq i\leq n-1,\\[1mm]
[e_1,x]=\sum\limits_{k=3}^na_{k-1}e_k,&
[e_2,x]=e_2,&
[e_i,x]=\sum\limits_{k=i+1}^na_{k+1-i}e_k, & 3\leq i\leq n,\\[1mm]
[x,e_2]=-e_2,& [x,x]=\gamma e_n,& [y_i,x]=\sum\limits_{k=i+1}^{n-1}a_{k+1-i}y, & 1\leq i\leq n-1.\\[1mm]
\end{array}\right.$$
Note that, the first non-vanishing parameter $\{a_2, a_3,\dots,a_{n-1}, \gamma\}$
in the algebras  $G_5$ and $G_6$ can be scaled to $1$.
\end{theorem}

\begin{proof}
The proof is similar to the proof of Theorem \ref{solvH(0)}.
\end{proof}

Now we consider the case when the nilradical is a non-split superalgebra from the class $G(\beta_4, \beta_5, \dots, \beta_n, \gamma),$ i.e., at least one of the parameters is not equal to zero.

\begin{theorem} Let $L=L_0\oplus L_1$ be a solvable Leibniz superalgebra whose nilradical is isomorphic to a non-split superalgebra $N$ from the class $G(\beta_4, \beta_5, \dots, \beta_n, \gamma).$ Then $L$ is isomorphic to one of the following pairwise non-isomorphic superalgebras:
$$\small SG_1({t}) (4\leq t\leq n):\left\{\begin{array}{lll}
[e_1,e_1]=e_3,& [e_i,e_1]=e_{i+1}, & 3\leq i\leq n-1,\\[1mm]
[y_j,e_1]=y_{j+1},& &1\leq j\leq n-2,\\[1mm]
[y_1,y_1]=e_1,& [y_j,y_1]=e_{j+1}, &2\leq j\leq n-1,\\[1mm]
[e_1,y_1]=\frac{1}{2}y_2,& [e_i,y_1]=\frac{1}{2}y_{i}, & 3\leq i\leq n-1,\\[1mm]
[e_1,e_2]=e_{t},&[e_j,e_2]=e_{j+t-2},&3\leq j\leq n-2,\\[1mm]
 [y_j,e_2]=y_{j+t-2},& &2\leq j\leq n-3,\\[1mm]
[y_i,x]=(2i-1)y_i,& & 1\leq i\leq n-1,\\[1mm]
 [x,y_1]=-y_1,& [e_1,x]=2e_1,& [x,e_1]=-2e_1, \\[1mm]
[e_2,x]=2(t-2)e_2,& [x,e_2]=-2(t-2)e_2-2e_{t-1},\\[1mm]
[e_i,x]=2(i-1)e_i,& &3\leq i\leq n, \\[1mm]
\end{array}\right.$$
$$\small SG_2(\gamma) (n \  \text{is odd}):\left\{\begin{array}{lll}
[e_1,e_1]=e_3,& [e_i,e_1]=e_{i+1}, & 3\leq i\leq n-1,\\[1mm]
[y_j,e_1]=y_{j+1},& &1\leq j\leq n-2,\\[1mm]
[y_1,y_1]=e_1,& [y_j,y_1]=e_{j+1}, &2\leq j\leq n-1,\\[1mm]
[e_1,y_1]=\frac{1}{2}y_2,& [e_i,y_1]=\frac{1}{2}y_{i}, & 3\leq i\leq n-1,\\[1mm]
[e_1,e_2]=e_{\frac{n+3}2},& [e_j,e_2]=e_{j+\frac{n-1}2},&3\leq j\leq n-2,\\[1mm]
 [y_j,e_2]=y_{j+\frac{n-1}2},& &1\leq j\leq n-3,\\[1mm]
[e_2,e_2]=\gamma e_n, \gamma\neq 0, &  [x,y_1]=-y_1,\\[1mm]
[y_i,x]=(2i-1)y_i,& & 1\leq i\leq n-1,\\[1mm]
 [e_1,x]=2e_1,& [x,e_1]=-2e_1, \\[1mm]
[e_2,x]=(n-1)e_2,& [x,e_2]=-(n-1)e_2-2e_\frac{n+1}2,\\[1mm]
[e_i,x]=2(i-1)e_i,& & 3\leq i\leq n. \\[1mm]
\end{array}\right.$$
$$SG_3:\left\{\begin{array}{lll}
[e_1,e_1]=e_3,& [e_i,e_1]=e_{i+1}, & 3\leq i\leq n-1,\\[1mm]
[y_j,e_1]=y_{j+1},& &1\leq j\leq n-2,\\[1mm]
[y_1,y_1]=e_1,& [y_j,y_1]=e_{j+1}, &2\leq j\leq n-1,\\[1mm]
[e_1,y_1]=\frac{1}{2}y_2,& [e_i,y_1]=\frac{1}{2}y_{i}, & 3\leq i\leq n-1,\\[1mm]
[e_2,e_2]=e_n,& [x,y_1]=-y_1,\\[1mm]
[y_i,x]=(2i-1)y_i,& & 1\leq i\leq n-1,\\[1mm]
[e_1,x]=2e_1,& [x,e_1]=-2e_1, \\[1mm]
[e_2,x]=(n-1)e_2,& [x,e_2]=-(n-1)e_2,\\[1mm]
[e_i,x]=2(i-1)e_i,& & 3\leq i\leq n, \\[1mm]
\end{array}\right.$$
\end{theorem}
The proof of this theorem is carried out similarly to the proof of Theorem \ref{solvH}.

\subsubsection*{Funding} This work was supported by the grant ''Automorphisms of operator algebras, classifications of infinite-dimensional non-associative algebras and superalgebras``, No. FZ-202009269, Ministry of Innovation Development of the Republic of Uzbekistan, Tashkent, Uzbekistan, 2021-2025.

\subsubsection*{Data Availability} The datasets generated during and/or analyzed during the current study are available from the corresponding author (Khudoyberdiyev A.Kh) on reasonable request.

\subsubsection*{Conflict of Interest} The authors have no conflicts and interests to declare that are relevant to the content of this article.

{\small\bibliography{commat}}

\EditInfo{May 25, 2023}{June 27, 2023}{Adam Chapman, Ivan Kaygorodov and Mohamed Elhamdadi}

\end{document}